\documentclass[reqno]{amsproc}
\usepackage{tikz}
\usepackage{amstext,amsmath,amssymb,amsfonts}
\usepackage[latin1]{inputenc}
\usepackage{epsfig}
\usepackage{hyperref}
\usepackage{color}

\usepackage{geometry}
\usepackage{amsthm}

\definecolor{myurlcolor}{rgb}{0,0,0.7}

\newtheorem{rmq}{Remark}[section]
\newtheorem{dfn}{Definition}[section]
\newtheorem{lem}{Lemma}[section]
\newtheorem{thm}{Theorem}[section]

\newcommand{\bprof}{\begin{prof}}
\newcommand{\eprof}{\end{prof}}
\newenvironment{prof}[1][Proof]{\textbf{#1.} }{\ \rule{0.5em}{0.5em}}
\usepackage{cite}

\newcommand{\bea}{\begin{eqnarray}}
\newcommand{\eea}{\end{eqnarray}}
\newcommand{\beq}{\begin{equation}}
\newcommand{\eeq}{\end{equation}}
\newcommand{\enn}{\nonumber \end{equation}}
\newcommand{\beqs}{\begin{eqnarray*}}
\newcommand{\eeqs}{\end{eqnarray*}}

 \newcommand{\cE}{\mathcal{E}}
\newcommand{\cT}{\mathcal{T}}

\newcommand{\curl}{\mathop{\rm curl}\nolimits}
\newcommand{\dive}{\mathop{\rm div}\nolimits}

\def\Om{\Omega}

\title[A posteriori error analysis]
{A posteriori error analysis for a new fully-mixed isotropic discretization of the stationary Stokes-Darcy coupled problem}

\author{ Hou\'edanou Koffi Wilfrid$^{(a)}$ and Adetola Jamal $^{(b)}$}
\email{a) khouedanou@yahoo.fr}
\address{D\'epartement de Math\'ematiques,
	Universit\'e d'Abomey-Calavi (UAC), Rep. of Benin}
\email{b) adetolajamal58@yahoo.com}
\address{Institut de Mathématiques et de Sciences Physiques (IMSP),
	Universit\'e d'Abomey-Calavi (UAC), Rep. of Benin}

\begin{document}

\maketitle
\begin{normalsize}
\begin{abstract}\normalsize
	In this paper we develop an a posteriori error analysis 
	for the stationary
	Stokes-Darcy coupled problem approximated by conforming finite element method on 
	isotropic meshes in $\mathbb{R}^d$, $d\in\{2,3\}$.
	The approach utilizes a new robust stabilized fully mixed discretization developed in \cite{JAFH:2018}.
	The a posteriori error estimate is based on a suitable evaluation on the residual of the finite 
	element solution plus the stabilization terms. It is proven that the a posteriori error estimate provided in this paper is both reliable and efficient.
	\\
	\small{\bf Mathematics Subject Classification [MSC]:} 74S05,74S10,74S15,
	74S20,74S25,74S30.
\end{abstract}
\tableofcontents
\section{Introduction}
There are many serious problems currently facing the world in which the coupling between groundwater and surface water is 
important. These include questions such as predicting how pollution discharges into streams, lakes, and rivers making its way into 
the water supply. This coupling is also important in technological applications involving filtration.
We  refer to the nice overview \cite{27} and the references therein for its physical background, modeling, and standard numerical 
methods. One important issue in the modeling of the coupled Darcy-Stokes flow is the treatement of the interface 
condition, where the  Stokes fluid meets the porous medium. In this paper, we only consider the so-called Beavers-Joseph-Saffman condition, which was experimentally derived by Beavers and Joseph in \cite{23}, modified by Saffman in \cite{44}, 
and later mathematically justified in \cite{32,17,48,37}.

There are three 
popular  formulations of the coupled Darcy-Stokes flow, namely  the primal formulation, the  mixed formulation in the Darcy region or the fully mixed formulation, see for examples \cite{21, 45,29,30,34,12,38,39,JAFH:2018} for some mathematical analysis. The authors in \cite{12} studied two different mixed 
formulations: the first one imposes the   weak continuity of the normal component of the velocity field on the interface, by using a 
Lagrange multiplier; while the second one imposes the  strong continuity  in the functional space.
Later on we   call these two mixed formulations, the weakly coupled formulation and the strongly
coupled formulation  respectively.
The weakly coupled formulation gives more freedom in the choice of   the discretization in the Stokes side and the Darcy side
separately. The works in \cite{7,45,29,30,12,AHN:15,HA:2016} are based on the weakly coupled formulation. Researches on the strongly coupled 
formulation have been focused on the development of an unified discretization,  that is, the Stokes side and
the Darcy side are discretized using the same finite element. This approach   simplifies the numerical implementation, 
  only if the unified discretization is not significantly more complicated than the commonly used discretizations for
the Darcy and the Stokes problems. 
In \cite{21,GL:2018}, a conforming, unified finite element has been proposed for the strongly coupled mixed formulation.
Superconvergence analysis of the finite element methods for the Stokes-Darcy system was studied in \cite{16}.
Other  less
restrictive 
discretizations   as the non-conforming unified approach \cite{7,34} or the discontinuous Galerkin  $(\textbf{DG})$ 
approach   have been proposed in \cite{33,38,39}. Due to its discontinuous nature, some $(\textbf{DG})$ discretizations for the coupled 
Darcy-Stokes problem may break the strong coupling in the discrete level \cite{38,39}, as they impose the normal continuity 
across the interface via interior penalties. 

A posteriori error estimators are computable quantities, expressed in terms of the discrete solution  and of the data that measure the actual discrete
errors without the knowledge of the exact solution. They are essential to design adaptive mesh 
refinement  algorithms  which equi-distribute the computational effort and optimize the approximation efficiency. 
Since the pioneering work of Babuska and Rheinboldt \cite{babuska:78a},   adaptive finite element methods based on 
a posteriori error estimates have been extensively investigated.


A posteriori error estimations have been well-established for both the mixed formulation of the Darcy flow \cite{Ca:97,24,35}, and the Stokes flow \cite{20,22,25,26,28,31,36,40,41,42}.
However,  only  few works exist for the coupled Darcy-Stokes problem, see for instance \cite{43,49,46,47,AHN:15}.
The paper \cite{43,AHN:15} concern the strongly coupled mixed formulation  where a $H(\dive)$ conforming and nonconforming finite element methods
have been used and \cite{46,49} concern the weakly coupled mixed formulation while \cite{46} uses the 
primal formulation  on the Darcy side. The authors in \cite{47} employ a fully-mixed formulation where Raviart-Thomas elements have been used to approximate the velocity in both the Stokes domain and Darcy domain, and constant piecewise  for approximate the pressure.

In \cite{JAFH:2018}, a stabilized finite element method for the stationary mixed Stokes-Darcy problem has been proposed for  the fully-mixed formulation.
The authors have used the well-know MINI elements ($P1b-P1$) to approximate the velocity and pressure in the conduit for Stokes equation. To capture the fully mixed technique in the porous medium region linear Lagrangian elements, $P1$ have been used for hydraulic (piezometric) head and Brezzi-Douglas-Marini ($BDM1$) piecewise constant finite elements have been used for Darcy velocity. 
An a priori error analysis is performed with some numerical tests confirming the convergence rates.
To our best knowledge, there is no a posteriori error estimation for the fully-mixed discretization proposed in \cite{JAFH:2018}.
Here we develop such a posteriori error analysis. 
The a posteriori error estimate is based on a suitable evaluation on the residual of the finite 
element solution. We further prove that our a posteriori error estimator is both reliable and efficient.
The difference between our paper and the   reference \cite{47} is that our discretization  uses 
MINI elements ($P1b-P1$) to approximate the velocity and pressure in the conduit for Stokes equations, $P1$-Lagrange elements to approximate hydraulic (piezometric) head and Brezzi-Douglas-Marini ($BDM1$) piecewise constant finite elements have been used for Darcy velocity. As a
result, additional term is included in the error estimator that measure the stability
of the method. In order to treat appropriately this stability term, we further need
a special Helmholtz decomposition \cite[Theorem 3.1]{AHN:15},
  a regularity result \cite[Theorem 3.2]{AHN:15} and an estimate of the stability error  
 \cite[Theorem 3.3]{AHN:15}.

The   paper  is organized as follows.
Some  preliminaries and  notation are given in  section  \ref{sec:R1}. 
The efficiency result is derived using the technique of bubble 
function introduced by R. Verf\"{u}rth \cite{verfurth:96b} and used in similar context by 
C. Carstensen \cite{Ca:97,carstensenandall}.
In section \ref{sec:R2}, the a posteriori error estimates are derived. 
\section{Preliminaries and Notations}\label{sec:R1}
\subsection{Model problem}
We consider the model of a flow in a bounded domain $\Omega\subset \mathbb{R}^d$ $(d=2 \mbox{  or  } 3)$, consisting of a 
porous medium domain $\Omega_p$, where the flow is a Darcy flow, and an  open region $\Omega_f=\Omega\smallsetminus 
\overline{\Omega}_p,$ where the flow is governed by the Stokes equations. The two regions are separated by an interface 
$\Gamma=\partial \Omega_p\cap \partial \Omega_f.$ Let $\Gamma_l=\partial \Omega_l\smallsetminus \Gamma$, $l=f,p$.
Each interface and boundary is assumed to be polygonal $(d=2)$ or polyhedral 
$(d=3)$. We denote by $\textbf{n}_f$ (resp. $\textbf{n}_p$) the 
unit outward normal vector along $\partial \Omega_f$  (resp. $\partial \Omega_p$). Note that on the interface 
$\Gamma$, we have $\textbf{n}_f=-\textbf{n}_s$. 
The Figure \ref{F1} shows a sketch of the problem domain, its boundaries and some other notations.
\begin{figure}[htpb]
 \centering
\begin{center}
	\tikzstyle{grisEncadre}=[thick, dashed, fill=gray!20]
	\begin{tikzpicture}[scale=0.85]
	color=gray!100;
	\draw [very thick](1,1)--(7,1);
	\draw [very thick](1,1)--(1,5.5);
	\draw [very thick](1,5.5)--(7,5.5);
	\draw [very thick](7,1)--(7,5.5);
	\draw [very thick](1,3)--(7,3);
	\draw [black,fill=gray!30] (1,1) -- (7,1) -- (7,3) --(1,3) -- cycle;
	\draw (3.7,1.6) node [above]{$\mbox{ \small\small $\Omega_p$:  \textbf{\small \small Porous Medium} }$};
	\draw (3.7,4) node [above]{$\mbox{  $\Omega_f$:  \textbf{\small \small Fluid Region} }$};
	\draw [>=stealth,->] [line width=1pt](2,3)--(2,3.5) node [right]{$\textbf{n}_p$};
	\draw [>=stealth,->] [line width=1pt](6,3)--(6,2.5) node [right]{$\textbf{n}_f$};
	\draw [>=stealth,->] [line width=1pt](4,2.8)--(4.7,2.8) node [below]{$\tau_j$};
	\draw  (4.5,3.8) node [below]{$\Gamma$};
	\draw[line width=0.5pt](1,1)--(1,3) node[midway,above,sloped]{$\Gamma_p$};
	\draw[line width=0.5pt](1,1)--(7,1) node[midway,below,sloped]{$\Gamma_p$};
	\draw[line width=0.5pt](7,1)--(7,3) node[midway,below,sloped]{$\Gamma_p$};
	\draw[line width=0.5pt](1,3)--(1,5.5) node[midway,above,sloped]{$\Gamma_f$};
	\draw[line width=0.5pt](1,5.5)--(7,5.5) node[midway,above,sloped]{$\Gamma_f$};
	\draw[line width=0.5pt](7,3)--(7,5.5) node[midway,below,sloped]{$\Gamma_f$};
	\end{tikzpicture}
\end{center}
\caption{\footnotesize{Global domain $\Omega$ consisting of the fluid region $\Omega_f$ and the porous media region $\Omega_p$ separated by the interface $\Gamma$.}}
\label{F1}
\end{figure}
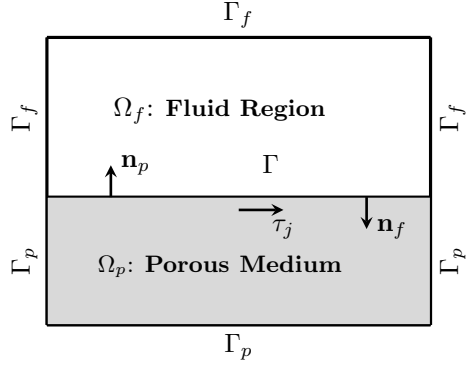\mbox{ }\\
The fluid velocity and pressure $\textbf{u}_f(x)$ and $p(x)$ are governed by the Stokes equations in $\Omega_f$:
\begin{eqnarray}\label{I1}
 \left\{
\begin{array}{ccccccccc}\label{r}
 -2\nu\nabla\cdot \mathbb{D}(\textbf{u}_f)+ \nabla p &=&\textbf{f}_f &\mbox{ in }&  &\Omega_f,&\\
\nabla\cdot\textbf{u}_f&=&0 &\mbox{  in   }&  &\Omega_f,&
\end{array}
\right.
\end{eqnarray}
where $\mathbb{T}=-p\mathbb{I}+2\nu\mathbb{D}(\textbf{u}_f)$ denotes the stress tensor, and 
$\mathbb{D}(\textbf{u}_f)=\frac{1}{2}\left(\nabla \textbf{u}_f+(\nabla\textbf{u}_f)^{T}\right)$ represents the deformation tensor. 
The porous media flow is governed by the following Darcy equations on $\Omega_p$ through the fluid velocity $\textbf{u}_p(x)$ and the piezometric head $\phi(x)$:
\begin{eqnarray}\label{I2}
 \left\{
\begin{array}{cccccccccccc}\label{rDarcy}
 \textbf{u}_p&=&-\textbf{K}\nabla \phi &\mbox{ in }&   &\Omega_p,&\\
\nabla\cdot \textbf{u}_p&=& f_p &\mbox{ in  }&  &\Omega_p.& \\
\end{array}
\right.
\end{eqnarray}
We impose impermeable boundary conditions, $\textbf{u}_p\cdot\textbf{n}_p=0$ on $\Gamma_p$, on the exterior boundary of the porous media region, and no slip conditions, $\textbf{u}_f=0$ on $\Gamma_f$, in the Stokes region. Both selections of boundary conditions can be modified. On $\Gamma$ the interface coupling conditions are conservation of mass, balance of forces and a tangential condition on the fluid region's velocity on the interface. The correct tangential condition is not competely understood (possibly due to matching a pointwise velocity in the fluid region with an averaged or homogenized velocity in the porous region). In this paper, we take the Beavers-Joseph-Saffman (-Jones), see \cite{32,17,48,37,44,23}, interfacial coupling:

\begin{eqnarray} \label{cd1}
\textbf{u}_f\cdot \textbf{n}_f+\mathbf{u}_p\cdot \textbf{n}_p&=&0 \mbox{  on  } \Gamma \\
\label{cd2}
-\textbf{n}_f\cdot\mathbb{T}\cdot\textbf{n}_f=
p-2\nu\textbf{n}_f\cdot\mathbb{D}(\textbf{u}_f)\cdot\textbf{n}_f&=&\rho g\phi\mbox{  on  } \Gamma\\
\label{cd3}
-\textbf{n}_f\cdot\mathbb{T}\cdot\tau_j=-2\textbf{n}_f\cdot\mathbb{D}(\textbf{u}_f)\cdot\tau_j&=&\frac{\alpha}{\sqrt{\tau_j\cdot\textbf{K}\tau_j}}\textbf{u}_f\cdot\tau_j, 1\leq j\leq (d-1) \mbox{  on  } \Gamma.
\end{eqnarray}
This is a simplification of the original and more physically relistic Beavers-Joseph conditions (in which $\textbf{u}_f\cdot\tau_j$ in (2.8) is replaced by $(\textbf{u}_f-\textbf{u}_p)\cdot\tau_j$; see \cite{23} ).
Here we denote 
\\
$\textbf{f}_f$, $f_p$-body forces in the fluid region and source in the porous region,\\
$\textbf{K}$-symmetric positive define (SPD) hydraulic conductivity tensor,\\
$\alpha$-constant parameter.

We shall also assume that all material and fluid parameters defined above are uniformly positive and bounded, i.e.,
$$0\leq k_{\min}\leq \lambda(\textbf{K})\leq k_{\max}< \infty.$$
\subsection{Notations and the weak formulation}
In this part, we first introduce some Sobolev spaces \cite{Adam:2003} and norms.
If $W$ is a bounded domain of $\mathbb{R}^d$ and $m$ 
is a non negative integer, the Sobolev space $H^m(W)=W^{m,2}(W)$ is 
defined in the usual way with the usual norm $\parallel\cdot\parallel_{m,W}$ and semi-norm $|.|_{m,W}$. In particular, 
$H^0(W)=L^2(W)$ and we write $\parallel\cdot\parallel_W$ for $\parallel\cdot\parallel_{0,W}$.
Similarly we   denote by
    $(\cdot,\cdot)_{W}$  the $L^2(W)$ $[L^2(W)]^N$ or $ [L^2W)]^{d\times d}$ inner product.
For shortness if $W$ is equal to $\Omega$, we will drop  the index $\Omega$, while  for any $m\geq 0$, 
$\parallel\cdot\parallel_{m,l}=\parallel\cdot\parallel_{m,\Omega_l}$, $|.|_{m,l}=|.|_{m,\Omega_l}$ 
and $(.,.)_l=(\cdot,\cdot)_{\Omega_l}$, for $l=f,s$.
The space  $H_0^m(\Omega)$ denotes the closure of $C_0^{\infty}(\Omega)$ in $H^{m}(\Omega)$. Let $[H^m(\Omega)]^d$ be the space of
vector valued functions $\textbf{v}=(v_1,\ldots,v_d)$ with components  $v_i$ in $H^m(\Omega)$. The 
norm and the seminorm on $[H^m(\Omega)]^d$ are given by 
\begin{eqnarray}
 \parallel\textbf{v}\parallel_{m,\Omega}:=\left(\sum_{i=0}^N\parallel v_i\parallel_{m,\Omega }^2\right)^{1/2} 
 \mbox{  and  } |\textbf{v}|_{m,\Omega}:= \left(\sum_{i=0}^N |v_i|_{m,\Omega}^2\right)^{1/2}.
\end{eqnarray}
For a connected open subset of the boundary $E\subset \partial \Omega_s\cup\partial \Omega_d$, we write 
$\langle.,.\rangle_{E}$ for the $L^2(E)$ inner product (or duality pairing), that is, for scalar valued functions $\lambda$, 
$\sigma$ one defines:
\begin{eqnarray}
 \langle\lambda,\sigma\rangle_{E}:=\int_{E} \lambda\sigma ds
\end{eqnarray}

By setting the space
$$H_{\dive}:=H(\dive;\Omega_p)=\left\{ \textbf{v}_p\in [L^2(\Omega_p)]^d: \nabla\cdot \textbf{v}_p\in L^2(\Omega_p)\right\},$$
we introduce the following spaces:
\begin{eqnarray*}
	\textbf{X}_f&:=&\left\{ \textbf{v}_f\in [L^2(\Omega_f)]^d: \textbf{v}_f=\textbf{0} \mbox{ on } \Gamma_f  \right\},\\
	Q_f&:=&L^2(\Omega_f),\\
	\textbf{X}_p&:=&\left\{\textbf{v}_p\in H(\dive; \Omega_p): \textbf{v}_p\cdot \textbf{n}_p=0 \mbox{  on } \Gamma_p \right\},\\
	Q_p&:=&L^2(\Omega_p).
\end{eqnarray*}
For the spaces $\textbf{X}_f$ and $\textbf{X}_p$, we define the following norms:
\begin{eqnarray*}
	\parallel \textbf{v}_f\parallel_1&:=& \sqrt{\parallel \textbf{v}_f\parallel_{\Omega_f}^2+|\textbf{v}_f|_{1,\Omega_f}^2}, 
	\mbox{   with  }  |\textbf{v}_f|_{1,\Omega_f}=\parallel\nabla\textbf{v}_f\parallel_{\Omega_f} \forall \textbf{v}_f\in \textbf{X}_f, \\
	\parallel\textbf{v}_p\parallel_{\dive}&:=&\sqrt{\parallel \textbf{v}_p\parallel_{\Omega_p}^2+\parallel\nabla\cdot\textbf{v}_p\parallel_{\Omega_p}^2}, \forall \textbf{v}_p\in\textbf{X}_p.
\end{eqnarray*}
The variational formulation of the steady-state Stokes-Darcy problem (\ref{I1})-(\ref{cd3}) reads as:
Find $(\textbf{u}_f,p;\textbf{u}_p,\phi)\in (\textbf{X}_f,Q_f;\textbf{X}_p,Q_p)$ satisfying:
\begin{eqnarray}\label{e1}
a_f(\textbf{u}_f,\textbf{v}_f)-b_f(\textbf{v}_f,p)+c_{\Gamma}(\textbf{v}_f,\phi)&=&(\textbf{f}_f,\textbf{v}_f)_{\Omega_f} \mbox{     }  \forall\textbf{v}_f\in \textbf{X}_f,\\\label{e2}
b_f(\textbf{u}_f,q)&=&0 \mbox{    }\forall  q\in Q_f,\\\label{e3}
a_p(\textbf{u}_p,\textbf{v}_p)-b_p(\textbf{v}_p,\phi)-c_{\Gamma}(\textbf{v}_p,\phi)&=&0 \mbox{   } \forall \textbf{v}_p\in\textbf{X}_p,\\\label{e4}
b_p(\textbf{u}_p,\psi)&=&\rho g (f_p,\psi)_{\Omega_p}  \mbox{  } \psi \in Q_p,
\end{eqnarray}
where the bilinear forms are defined as:
\begin{eqnarray*}
	a_f(\textbf{u}_f,\textbf{v}_f)&:=&2\nu (\mathbb{D}(\textbf{u}_f),\mathbb{D}(\textbf{v}_f))_{\Omega_f}+
	\displaystyle\sum_{j=1}^{d-1}\frac{\alpha}{\sqrt{\tau_j\cdot\textbf{K}\tau_j}}\left<\textbf{u}_f\cdot\tau_j,\textbf{v}_f\cdot\tau_j\right>_{\Gamma}\\
	a_p(\textbf{u}_p,\textbf{v}_p)&:=&\rho g(\textbf{K}^{-1}\textbf{u}_p,\textbf{v}_p)_{\Omega_p},\\
	b_f(\textbf{v}_f,p)&:=&(p,\nabla\cdot\textbf{v}_f)_{\Omega_f},\\
	b_p(\textbf{v}_p,\phi)&:=&\rho g(\phi,\nabla\cdot\textbf{v}_p)_{\Omega_p},\\
	c_{\Gamma}(\textbf{v}_f,\phi)&:=&\rho g\left<\phi,\textbf{v}_f\cdot\textbf{n}_f\right>_{\Gamma}.
\end{eqnarray*}
After introducing, for $\textbf{U}=(\textbf{u}_f,p,\textbf{u}_p,\phi)\in \textbf{X}_f\times Q_p\times \textbf{X}_p\times Q_p=\textbf{H}$ and $\textbf{V}=(\textbf{v}_f,q,\textbf{v}_p,\psi)\in \textbf{X}_f\times Q_p\times \textbf{X}_p\times Q_p$,
\begin{eqnarray}
\mathcal{L}(\textbf{U},\textbf{V})&:=&
a_f(\textbf{u}_f,\textbf{v}_f)-b_f(\textbf{v}_f,p)+	b_f(\textbf{u}_f,q) \\\nonumber
       &+&
	a_p(\textbf{u}_p,\textbf{v}_p)-	b_p(\textbf{v}_p,\phi)+
		b_p(\textbf{u}_p,\psi)+	c_{\Gamma}(\textbf{v}_f-\textbf{v}_p,\phi),\\
		\mathcal{F}(\textbf{V})&:=&(\textbf{f}_f,\textbf{v}_f)_{\Omega_f}+\rho g(f_p,\psi)_{\Omega_p},
\end{eqnarray}
the weak formulation (\ref{e1})-(\ref{e4}) can be equivalently rewritten as follows: Find 
$\textbf{U}\in \textbf{H}$ satisfying
\begin{eqnarray}\label{Fequivalente}
\mathcal{L}(\textbf{U},\textbf{V})=\mathcal{F}(\textbf{V}), \forall \textbf{V}\in\textbf{H}.
\end{eqnarray}
It is easy to verify that this variational fromulation is well-posedness.

To end this section, we recall the following Poincar\'e, Korn's and the trace inequalities, which will be used in the later analysis; There exist constant $C_p$, $C_K$, $C_v$, only depending on $\omega_f$ such that for all $\textbf{v}_f\in\textbf{X}_f$, 
$$\parallel\textbf{v}_f\parallel\leq C_p|\textbf{v}_f|_1, \mbox{   }
|\textbf{v}_f|\leq C_K\parallel\mathbb{D}(\textbf{v}_f)\parallel_{\Omega_f}, \mbox{  } 
\parallel \textbf{v}_f\parallel_{\Gamma} \leq C_v\parallel\textbf{v}_f\parallel_{\Omega_f}^{1/2}
|\textbf{v}_f|_{1,\Omega_f}^{1/2}.$$
Besides, there exists a constant $\tilde{C}_v$ that only depends on $\Omega_p$ such that for all 
$\psi\in Q_p$,
\begin{eqnarray}
\parallel \psi\parallel_{L^2(\Gamma)}\leq \tilde{C}_v\parallel\psi\parallel_{\Omega_p}^{1/2}
|\psi |_{1,\Omega_p}^{1/2}.
\end{eqnarray}
\subsection{Fully-mixed isotropic discretization}
First, we consider the family of triangulations $\cT_h$ of $\Omega$, consisting of $\cT_h^f$ and $\cT_h^p$, which are  regular triangulations of $\Omega_f$ and $\Omega_p$, respectively, where $h> 0$ is a positive parameter. We also assume that on the interface $\Gamma$ the two meshes of $\cT_h^f$ and $\cT_h^p$, which form the regular triangulation $\cT_h:=\cT_h^f\cup\cT_h^p,$ coincide. 

The domain of the uniformly regular triangulation $\overline{\Omega}_f\cup\overline{\Omega}_p$ is such that $\overline{\Omega}=\left\{\cup K: K\in\cT_h\right\}$ and $h=\displaystyle\max_{K\in\cT_h} h_K$. There exist positive constants $c_1$ and $c_2$ satisfying $c_1h\leq h_K\leq c_2\rho_K$. To approximate the diameter $h_K$ of the triangle (or tetrahedral) $K$, $\rho_K$ is the diameter of the greatest ball included in $K$. Based on the subdivisions $\cT_h^f$ and $\cT_h^p$, we can define finite element spaces $\textbf{X}_{fh}\subset \textbf{X}_h$, $Q_{fh}\subset Q_f$, $\textbf{X}_{ph}\subset \textbf{X}_p$, $Q_{ph}\subset Q_p$. We consider the well-known MINI elements $(P1b-P1)$ to approximate the velocity and the pressure in the conduit for Stokes equations \cite{ABF:1984}. To capture the fully-mixed technique in the porous medium region linear Lagrangian elements, $P1$ are used for hydraulic (piezometric) head and Brezzi-Douglas-Marini ($BDM1$) piecewise constant finite elements are used for Darcy velocity \cite{BDM:1985}. \\
In the fluid region, we select for the Stokes problem the finite element spaces $(\textbf{X}_{fh},Q_{fh})$ that satisfy the velocity-pressure inf-sup condition:
\emph{There exists a constant $C_f> 0$, independent of $h$, such that,}
\begin{eqnarray}
	\displaystyle\inf_{0\neq q^h\in Q_{fh}}\displaystyle\sup_{\textbf{0}\neq \textbf{v}_f^h\in \textbf{X}_{fh}}\frac{b_f(\textbf{v}_f^h,q^h)}{|\textbf{v}_f^h|_{1,\Omega_f}\parallel q_f^h\parallel_{\Omega_f}}\geq C_f.
\end{eqnarray}
In the porous region, we use the finite element spaces $(\textbf{X}_{ph},Q_{ph})$ that also satisfy a standard inf-sup condition:\emph{ There exist a constant $C_p> 0$ such that for all $\phi^h\in Q_{ph}$}, 
\begin{eqnarray}
\displaystyle\inf_{0\neq \phi^h\in Q_{ph}}\displaystyle\sup_{\textbf{0}\neq \textbf{v}_p^h\in \textbf{X}_{ph}}\frac{b_p(\textbf{v}_p^h,\phi^h)}{\parallel\textbf{v}_f^h\parallel_{\dive}\parallel \phi^h\parallel_{\Omega_p}}\geq C_p.
\end{eqnarray}

Then  the finite element discretization of (\ref{Fequivalente}) is to  find 
$\textbf{U}_h\in \textbf{H}_h=\textbf{X}_{fh}\times Q_{fh}\times \textbf{X}_{ph}\times Q_{ph}$ such that 
\begin{eqnarray}\label{discrete}
 \mathcal{L}(\textbf{U}_h,\textbf{V}_h)+\textbf{J}_{\Gamma}(\textbf{U}_h,\textbf{V}_h)=\mathcal{F}(\textbf{V}_h) \mbox{  } \forall \textbf{V}_h\in \textbf{H}_h.
\end{eqnarray}
This is the natural discretization of the weak formulation (\ref{Fequivalente}) except that  the stabilized term 
$\textbf{J}_{\Gamma}(\textbf{U}_h,\textbf{V}_h)$ is added. This bilinear form $\textbf{J}_{\Gamma}(.,.)$ is defined by 
\begin{eqnarray}
\textbf{J}_{\Gamma}(\textbf{U}_h,\textbf{V}_h):=\frac{\delta}{h}
\left<(\textbf{u}_f^h-\textbf{u}_p^h)\cdot\textbf{n}_f,(\textbf{v}_f^h-\textbf{v}_p^h )\cdot \textbf{n}_f\right>_{\Gamma}, \mbox{  } 0< h < 1.
\end{eqnarray}
We are now able to define the norm on $\textbf{H}_h$:
\begin{eqnarray*}
 \parallel\textbf{V}\parallel_h:=\sqrt{
 \parallel\textbf{v}_f^h\parallel_{1,\Omega_f}^2+\parallel q_f^h\parallel_{\Omega_f}^2+
 \parallel \textbf{v}_p^h\parallel_{\dive}^2+\parallel \psi^h\parallel_{\Omega_p}^2+h^{-1}\parallel (\textbf{v}_f^h-\textbf{v}_p^h)\parallel_{\Gamma}^2}
\end{eqnarray*}
We have the following results (see \cite[Theorem 2 and  Theorem 3]{JAFH:2018}):
\begin{thm}
There exists a unique solution $\textbf{U}_h\in \textbf{H}_h$ to   problem (\ref{discrete}) and 
 if  the solution $\textbf{U}\in \textbf{H}$ of the continuous problem (\ref{Fequivalente}) 
 is smooth enough, then we have:  
\begin{eqnarray}
  \parallel \textbf{U}-\textbf{U}_h\parallel_h\leq C(\textbf{U}) h.
 \end{eqnarray}
\end{thm}

Below, in order to avoid excessive use of constants, the abbreviation $x\lesssim y$ stand for $x\leqslant c y$, with
$c$ a positive constant independent of $x$, $y$ and $\mathcal{T}_h$.

 \begin{rmq}(\textbf{Galerkin orthogonality relation})
 Let $\textbf{U}=(\textbf{u}_f,p,\textbf{u}_p,\phi)\in \textbf{H}$ be the exact solution and 
 $\textbf{U}_h=(\textbf{u}_{fh},p_h,\textbf{u}_{ph},\phi_h)\in\textbf{H}_h$ be the finite element solution. Then for any 	
 $ \textbf{V}_h=(\textbf{v}_{fh},p_h,\textbf{v}_{ph},\psi_h)\in \textbf{H}_h$, and using technical regularity result Theorem \ref{treg} below,  we can subtract (\ref{Fequivalente}) to  (\ref{discrete}) to obtain the 
 Galerkin  orthogonality  relation:
 \begin{eqnarray*}\nonumber
 \mathcal{L}_h(\textbf{U}-\textbf{U}_h,\textbf{V}_h)&=&\mathcal{L}_h(\textbf{U},\textbf{V}_h)-
 \mathcal{L}_h(\textbf{U}_h,\textbf{V}_h)\\
 &=&\mathcal{L}(\textbf{U},\textbf{V}_h)-\mathcal{L}_h(\textbf{U}_h,\textbf{V}_h)\\
 &=&\mathcal{F}(\textbf{V}_h)-\mathcal{F}(\textbf{V}_h)\\
 &=&0.
 \end{eqnarray*}
 Thus, we have the relation:
 \begin{eqnarray*}
 2\nu\left(\mathbb{D}(\textbf{e}_f),\mathbb{D}(\textbf{v}_{fh})\right)_{\Omega_f}&+&
 \displaystyle\sum_{j=1}^{d-1}\frac{\alpha}{\sqrt{\tau_j\cdot\textbf{K}\cdot\tau_j}}
 \left<\textbf{e}_f\cdot\tau_j,\textbf{v}_{fh}\cdot\tau_j\right>_{\Gamma}-
 \left(\epsilon_p,\nabla\cdot\textbf{v}_{fh}\right)_{\Omega_f}+
 \left(q_h,\nabla\cdot\textbf{e}_f\right)_{\Omega_f}\\
 &+&\rho g\left[\left(\textbf{K}^{-1}\textbf{e}_p,\textbf{v}_{ph}\right)_{\Omega_p}-(\lambda_{\phi},\nabla\cdot\textbf{v}_{ph})_{\Omega_p}+(\psi_h,\nabla\cdot\textbf{e}_p)_{\Omega_p}+
 \left< \lambda_{\phi},[\textbf{v}]\right>_{\Gamma}\right]\\
 &=&0,
 \end{eqnarray*}

where here and below, the errors in the velocity and in the pressure of Stokes equations, and errors in the hydraulic and Darcy velocity equations are respectively defined by:
\begin{eqnarray*}
\textbf{e}_f:=\textbf{u}_f-\textbf{u}_{fh}, \hspace*{0.3cm} \epsilon_p:= p-p_h, \hspace*{0.2cm} 
\textbf{e}_p:=\textbf{u}_p-\textbf{u}_{ph}\mbox{  and   }  \lambda_{\phi}=\phi-\phi_h.
\end{eqnarray*}
\end{rmq}

\section{A posteriori error analysis}\label{sec:R2}
\subsection{Some technical results}
Our a posteriori analysis requires some  analytical results that are recalled. We define the space
$$\mathcal{H}=\left\{\textbf{v}\in \textbf{H}(\dive,\Omega): \textbf{v}_{|\Omega_f}\in \textbf{X}_f \mbox{  and } \textbf{v}_{|\Omega_p}\in \textbf{X}_p \right\}$$ with the norm
$$\parallel \textbf{v}\parallel_{\mathcal{H}}:=\sqrt{|\textbf{v}_f|_{1,\Omega_f}^2+\parallel\textbf{v}_p\parallel_{\Omega_p}^2+\parallel\nabla\cdot\textbf{v}_p\parallel_{\Omega_p}}.$$
The first one concerns a sort of Helmholtz decomposition  of elements of $\mathcal{H}$.
Recall first that if $d=3$,
\[
H_0(\curl, \Omega_p)=\{\psi\in L^2(\Omega_p)^3: \curl \psi\in L^2(\Omega_p)^3 \hbox{ and }
\psi\times \textbf{n}=\textbf{0}\hbox{ on } \partial \Omega_p\}.
\]
\begin{thm}(Ref. \cite[Page 708]{AHN:15})\label{thelmholtz}
	Any $\textbf{v}\in \mathcal{H}$ admits the Helmholtz type decomposition 
	\begin{equation}\label{helmholtz}
	\textbf{v}=\textbf{v}_0+\textbf{v}_1,
	\end{equation}
	where $\textbf{v}_0,\textbf{v}_1 \in \mathcal{H}$ but satisfying $\textbf{v}_0\in H^1(\Omega)^d$,
	\beq\label{defv1}
	\textbf{v}_1=
	\left\{
	\begin{array}{ccc}
		\textbf{0}& \hbox{ in } &\Omega_f,\\
		\curl \beta_p & \hbox{ in } &\Omega_p,
	\end{array}
	\right.
	\eeq
	where $\beta_p\in H^1_0(\Omega_p)$ if $d=2$, while 
	$\beta_p\in H^1(\Omega_p)^3\cap H_0(\curl, \Omega_p)$ if $d=3$, with the estimate
	\begin{equation}\label{helmholtzest}
	\|\textbf{v}_0\|_{1, \Omega}+\|\beta_p\|_{1, \Omega_p}\lesssim \|\textbf{v}\|_{\mathcal{H}}.
	\end{equation}
\end{thm}
The second result that we need is a regularity result for the solution $\textbf{U}=(\textbf{u}_f,p,\textbf{u}_p,\phi)\in\textbf{H}$ of (\ref{Fequivalente}) is the following theorem:

\begin{thm}(\cite[Page 710]{AHN:15})\label{treg}
	Let $\textbf{U}\in 
	\textbf{H}$ be the unique solution  of (\ref{Fequivalente}).
	If $\textbf{K}\in [C^{0,1}(\bar \Om_p)]^{d\times d}$, then 
	there exists $\epsilon>0$ such that
	\[
	\textbf{u}_{|\Omega_p}\in [H^{\frac12+\epsilon}(\Om_p)]^d.
	\]
\end{thm}

Let us finish this section by  an estimation of the stability error (see \cite[Theorem 3.3]{AHN:15}):
\begin{thm}
	\label{tnonconferror}
	For any $ \textbf{U}_h=(\textbf{u}_{fh},p_h,\textbf{u}_{ph},\phi_h) \in  \textbf{H}_h $ we have
	\begin{equation}\label{nonconferror}
	\inf_{\textbf{W}_h\in \textbf{H}_h\cap \mathcal{H}} \|\textbf{U}_h-\textbf{W}_h \|_h^2
	\lesssim \textbf{J}_{\Gamma}(\textbf{U}_h,\textbf{U}_h).
	\end{equation}
\end{thm}

\subsection{Error estimator}In order to  solve the Stokes-Darcy coupled problem by efficient adaptive finite element methods, reliable and efficient 
a posteriori error analysis is important to provide appropriated indicators. 
In this section, we first define the local and global indicators and then the lower and upper error  bounds are derived in Section \ref{mainresul}. 
\subsubsection{Error equations}
The general philosophy of residual error estimators is to estimate an appropriate norm of the correct residual by terms that 
can be  evaluated easier, and that involve the data at hand. Thus we define the error equations:
Let $\textbf{U}=(\textbf{u}_f,p,\textbf{u}_p,\phi)\in \textbf{H}$ be the exact solution and 
$\textbf{U}_h=(\textbf{u}_{fh},p_h,\textbf{u}_{ph},\phi_h)\in\textbf{H}_h$ be the finite element solution. Then for any $\textbf{V}_h=(\textbf{v}_{fh},q_h,\textbf{v}_{ph},\psi_h)\in\textbf{H}_h$ and 
$\textbf{V}=(\textbf{v}_f,p,\textbf{v}_p,\psi)\in\textbf{H}$, using the Helmholtz decomposition (Theorem \ref{thelmholtz}) we have: 
\begin{eqnarray} \label{errorequation}
\mathcal{L}_h(\textbf{U}-\textbf{U}_h,\textbf{V})&=&\mathcal{L}_h(\textbf{U}-\textbf{U}_h,\textbf{V}-\textbf{V}_h)\\\nonumber
&=&\left[\displaystyle\sum_{K\in\cT_h^f}\left(\textbf{R}_K^f(\textbf{U}_h),\textbf{V}-\textbf{V}_h)\right)_{K}+\displaystyle\sum_{K\in\cT_h^p}\left(\textbf{R}_K^p(\textbf{U}_h),\textbf{V}-\textbf{V}_h)\right)_{K}\right],
\end{eqnarray}
where
\begin{eqnarray}\label{errorequations}
\left(\textbf{R}_K^f(\textbf{U}_h),\textbf{V}-\textbf{V}_h\right)_K&=&
\left(\textbf{f}_{f}+2\nu\nabla\cdot\mathbb{D}(\textbf{u}_{fh})-\nabla p_h,\textbf{v}_{f}-\textbf{v}_{fh}\right)_K-\left(q-q_h,\nabla\cdot\textbf{u}_{fh}\right)_K\\\nonumber
&-&\left[\displaystyle\sum_{E\in\cE_h(\partial K\cap \Gamma)}\displaystyle
\sum_{j=1}^{d-1}\left(2\nu\textbf{n}_f\cdot\mathbb{D}(\textbf{u}_{fh})\cdot\tau_j+\frac{\alpha}{\sqrt{\tau_j\cdot\textbf{K}\cdot\tau_j}}\textbf{u}_{fh}\cdot\tau_j,(\textbf{v}_{f}-\textbf{v}_{fh})\cdot\tau_j\right)_E\right]\\\nonumber
&+&\displaystyle\left[\sum_{E\in\cE_h(\partial K\cap \Gamma)}\left(
p_h-2\nu\textbf{n}_f\cdot\mathbb{D}(\textbf{u}_{fh})\cdot\textbf{n}_{f}-\rho g \phi_h,(\textbf{v}_f-\textbf{v}_{fh})\cdot\textbf{n}_f\right)_E\right],
\end{eqnarray}
and 
\begin{eqnarray}\nonumber
\left(\textbf{R}_K^p(\textbf{U}_h),\textbf{V}-\textbf{V}_h\right)_K&=&
\left(\curl (\rho g\textbf{K}^{-1}\textbf{u}_{ph}+\nabla\phi_h),\beta_p-\beta_{ph}\right)_K\\\nonumber
&+&\left(\rho g(f_p-\nabla\cdot\textbf{u}_{ph}),\psi-\psi_h\right)_K\\\nonumber\label{errorequationd}
&-&\displaystyle\sum_{E\in\cE_h(\partial K\cap \Omega_p)}
\left(\rho g \textbf{K}^{-1}\textbf{u}_{ph}+\nabla\phi_h)\times \textbf{n}_E,\psi-\psi_h\right)_E\\
&+&\displaystyle\sum_{E\in\cE_h(\partial K\cap \Omega_p)}\left([\rho g\phi_h\textbf{n}_E]_E,\beta_p-\beta_{ph}\right)_E\\\nonumber
&-&\displaystyle\sum_{E\in\cE_h(\partial K\cap \Gamma)}\left([\rho g\phi_h\textbf{n}_E]_E,[\textbf{v}-\textbf{v}_h]_E\right)_E.
\end{eqnarray}

\subsubsection{Residual Error Estimators}
\begin{dfn}[\textbf{A posteriori error indicators}]
	The residual error estimator is locally defined by:
	\begin{equation}
		\Theta_K=\left[\Theta_{K,f}^2+\Theta_{K,p}^2\right]^{\frac{1}{2}} \mbox{  for each  } K\in\cT_h,
	\end{equation}
where 
\begin{eqnarray}\nonumber\label{Indf}
	\Theta_{K,f}^2&=&h_K^2\parallel \textbf{f}_{fh}+2\nu\nabla\cdot\mathbb{D}(\textbf{u}_{fh})-\nabla p_h\parallel_K^2+\parallel \nabla\cdot \textbf{u}_{fh}\parallel_K^2\\
	&+&\displaystyle\sum_{E\in\cE_h(\partial K\cap\Gamma)}h_E\left\{ \displaystyle\sum_{j=1}^{d-1}
	\parallel 2\nu\textbf{n}_f\cdot\mathbb{D}(\textbf{u}_{fh})\cdot\tau_j+\frac{\alpha}{\sqrt{\tau_j\cdot\textbf{K}\cdot\tau_j}}\textbf{u}_{fh}\cdot\tau_j\parallel_E^2\right\}\\\nonumber
	&+&\displaystyle\sum_{E\in\cE_h(\partial K\cap\Gamma)}h_E\parallel p_h-2\nu\textbf{n}_f\cdot\mathbb{D}(\textbf{u}_{fh})\cdot\textbf{n}_f-\rho g\phi_h\parallel_E^2
\end{eqnarray}
and 
	\begin{eqnarray}\nonumber\label{Indp}
		\Theta_{K,p}^2&=& h_K^2\parallel \curl (\rho g\textbf{K}^{-1}\textbf{u}_{ph}+\nabla \phi_h)\parallel_K^2+\parallel \rho g(f_p-\nabla\cdot\textbf{u}_{ph})\parallel_K^2\\
		&+&\displaystyle\sum_{E\in\cE_h(\partial K\cap \Omega_p)}\parallel 
		[\rho g(\textbf{K}^{-1}\textbf{u}_{ph}+\nabla\phi_h)\times \textbf{n}_p]_E\parallel_E^2\\\nonumber
		&+&\displaystyle\sum_{E\in\cE_h(\partial K\cap \overline{\Omega}_p)}h_E\parallel[\rho g\phi_h\textbf{n}_p]_E\parallel_E^2+
		\displaystyle\sum_{E\in\cE_h(\partial K\cap \Gamma)}\frac{\delta h_E}{h}\parallel[(\textbf{u}_{fh}-\textbf{u}_{ph})\cdot\textbf{n}_f]_E\parallel_E^2.
	\end{eqnarray}

The global residual error estimator is given by: 
\begin{eqnarray}
\Theta:=\left[\sum_{K\in\mathcal{T}_h}\Theta_K^2\right]^{\frac{1}{2}}.
\end{eqnarray}
Furthermore denote the local and global approximation terms by 
\begin{eqnarray*}
	\zeta_{K}:=
	\left\{
	\begin{array}{cc}
		h_{K}\parallel \textbf{f}_f-\textbf{f}_{fh}\parallel_{K}
		&\forall K\in \mathcal{T}_h^f,
		\\
	\rho g\parallel f_p-f_{ph}\parallel_K
		&\forall K\in \mathcal{T}_h^p,
	\end{array}
	\right.
\end{eqnarray*}
and 
\begin{eqnarray}
\zeta:=\left[\sum_{K\in\mathcal{T}_h}\zeta_K^2\right]^{\frac{1}{2}},
\end{eqnarray}
where the global function $\textbf{f}_{fh}: \Omega_f\rightarrow \mathbb{R}^d$ is defined by:
$${\textbf{f}_{fh}}_{|K}=\textbf{f}_K=\frac{1}{|K|}\int_K \textbf{f}_f(x)dx \mbox{   }\forall K\in\cT_h^f,$$
while in $\Omega_p$, we take ${f_h}_{|K}=f_K$ for all $K\in \cT_h^p$, as the unique element of $\mathbb{P}^1(K)$ such that:
$$\int_K f_K(x)q(x)dx=\int_K f(x)q(x)dx\mbox{   } \forall q\in \mathbb{P}^1(K).$$
\end{dfn}

\begin{rmq}
	The residual character of each term on the right-hand sides of (\ref{Indf}) and (\ref{Indp}) is quite clear
	since if $(\textbf{u}_{fh}, p_h,\textbf{u}_{ph},\phi_h)\in\textbf{H}_h$ would be the exact solution of (\ref{Fequivalente}), then they would vanish.
\end{rmq}

\subsubsection{Analytical tools}
\begin{enumerate}
\item \textbf{Inverse inequalities:} In order to derive the lower error bounds, we proceed similarly as in \cite{Ca:97} and 
\cite{carstensenandall} (see also \cite{15}), by applying
inverse inequalities, and the localization technique based on simplex-bubble and face-bubble functions. To this end, we 
recall some notation and introduce further preliminary results. Given $K\in \mathcal{T}_h$, and 
$E\in \cE(K)$,
we let $b_K$ and $b_E$ be the usual simplexe-bubble and face-bubble 
functions respectively (see (1.5) and (1.6) in \cite{verfurth:96b}). In particular, $b_K$ satisfies 
$b_K\in \mathbb{P}^3(K)$, $supp(b_K)\subseteq K$, $b_K=0 \mbox{ on } \partial K$, and $0\leq b_K\leq 1\mbox{ on } K $.
Similarly, $b_E\in \mathbb{P}^2(K)$, $supp(b_E)\subseteq 
\omega_E:=\left\{K'\in \mathcal{T}_h:  E\in\cE (K')\right\}$, 
$b_E=0\mbox{  on  } \partial K\smallsetminus E$ and $0\leq b_E\leq 1\mbox{ in } \omega_E$.
We also recall from \cite{verfurth:94a} that, given $k\in\mathbb{N}$, there exists an extension operator
$L: C(E)\longrightarrow C(K)$ that satisfies $L(p)\in \mathbb{P}^k(K)$ and $L(p)_{|E}=p, \forall p\in \mathbb{P}^k(E)$.
A corresponding vectorial version of $L$, that is, the componentwise application of $L$, is denoted by 
$\textbf{L}$. Additional properties of $b_K$, $b_E$ and $L$ are collected in the following lemma (see \cite{verfurth:94a})

\begin{lem}
	Given $k\in \mathbb{N}^*$, there exist positive constants depending only on $k$ and shape-regularity of the triangulations 
	(minimum angle condition), such that for each simplexe $K$ and $E\in \cE(K)$ there hold
	\begin{eqnarray}\label{cl1}
	\parallel q \parallel_{K}&\lesssim&\parallel qb_K^{1/2}\parallel_{K}\lesssim
	\parallel q\parallel_{K}, \forall q\in \mathbb{P}^k(K)\\\label{cl2}
	|q b_K|_{1,K}&\lesssim&  h_K^{-1}\parallel q \parallel_{K}, \forall q\in \mathbb{P}^k(K)\\\label{cl3}
	\parallel p\parallel_{E}&\lesssim&\parallel b_E^{1/2}p\parallel_{E}\lesssim \parallel p\parallel_{E},
	\forall p\in \mathbb{P}^k(E)\\\label{cl4}
	\parallel L(p)\parallel_{K} +h_E|L(p)|_{1,K}&\lesssim& h_E^{1/2}\parallel p\parallel_{E}
	\forall p\in \mathbb{P}^k(E)
	\end{eqnarray}
\end{lem}
\item \textbf{Continuous trace inequality}
\begin{lem}\label{lemd1}
	(Continuous trace inequality)
	There exists a positive constant  $\beta_1> 0$ depending only on  $\sigma_0$ such that 
	\begin{eqnarray}
	\parallel \textbf{v}\parallel_{\partial K}^2&\leqslant &\beta_1 \parallel\textbf{v}\parallel_{K}
	\parallel\textbf{v}\parallel_{1,K}, \mbox{  } \forall K\in \mathcal{T}_h, \forall \textbf{v}\in [H^1(K)]^d.
	\end{eqnarray}
\end{lem}
\item \textbf{Cl\'ement interpolation operator:} In order to derive the upper error bounds, 
we introduce the  Cl\'ement interpolation operator 
$\mbox{I}_{\mbox{Cl}}^0: H_0^1(\Omega)\longrightarrow \mathcal{P}_c^b(\mathcal{T}_h)$ that approximates optimally non-smooth 
functions by continuous piecewise linear functions:
\begin{eqnarray*}
	\mathcal{P}_c^b(\mathcal{T}_h):=\left\{v\in C^0(\overline{\Omega}): \mbox{   }
	v_{|K} \in \mathbb{P}^1(K), \mbox{   } \forall K\in\mathcal{T}_h \mbox{  and  } 
	v=0 \mbox{  on  } \partial \Omega\right\}
\end{eqnarray*}
In addition, we will make use of a vector valued version 
of $\mbox{I}_{\mbox{Cl}}^0$, that is, $\textbf{I}_{\mbox{Cl}}^0: [H_0^1(\Omega)]^d\longrightarrow [\mathcal{P}_c^b(\mathcal{T}_h)]^d $, which 
is defined componentwise by $\mbox{I}_{\mbox{Cl}}^0.$ The following lemma establishes the local approximation properties of 
$\mbox{I}_{\mbox{Cl}}^0$ (and hence of $\textbf{I}_{\mbox{Cl}}^0$), for a proof see \cite[Section 3]{clement:75}.
\begin{lem}\label{lem1}
	There exist constants $C_1, C_2> 0$, independent of $h$, such that for all $v\in H_0^1(\Omega)$ there hold 
	\begin{eqnarray}
	\parallel v-\mbox{I}_{Cl}^0(v)\parallel_{K}&\leq& C_1h_K\parallel v\parallel_{1,\Delta(K)} 
	\mbox{   } \forall K\in \mathcal{T}_h,   \hspace*{0.2cm}\mbox{  and  }\\
	\parallel v-\mbox{I}_{Cl}^0(v)\parallel_{E}&\leq& C_2h_E^{1/2}\parallel v\parallel_{1,\Delta(E)}
	\mbox{   } \forall E\in\cE_h,
	\end{eqnarray}
	where $\Delta(K):=\cup\left\{K'\in \mathcal{T}_h: K'\cap K\neq\emptyset\right\}$ and  
	$\Delta (E):=\cup \left\{K'\in \mathcal{T}_h: K'\cap E\neq\emptyset\right\}$.
	\end{lem}
\end{enumerate}

\subsection{Optimality of $\left\{\Theta_K\right\}_{K\in\cT_h}$}\label{mainresul}
\subsubsection{Reliability result}
\begin{thm} [\textbf{Reliability of $\Theta$}]\mbox{ }\\
	Let $\textbf{U}=(\textbf{u}_f,p,\textbf{u}_p,\phi)\in \textbf{H}$ be the exact solution of (\ref{Fequivalente}) and 
	$\textbf{U}_h=(\textbf{u}_{fh},p_h,\textbf{u}_{ph},\phi_h)\in\textbf{H}_h$ be the finite element solution of (\ref{discrete}). There exist a constant $C_{rel}> 0$ such that the following  estimate holds:
	\begin{equation}\label{reliability}
	\parallel\textbf{U}-\textbf{U}_h\parallel_h \leq C_{rel}\left(\Theta+\zeta\right).
	\end{equation}
\end{thm}
\begin{proof}
	We take $\psi_h=0=q_h$ in error equation (\ref{errorequation})-(\ref{errorequationd}) and we obtain:
	
	\begin{eqnarray} \label{errorequation1}
	\mathcal{L}_h(\textbf{U}-\textbf{U}_h,\textbf{V})&=&\mathcal{L}_h(\textbf{U}-\textbf{U}_h,\textbf{V}-\textbf{V}_h)\\\nonumber
	&=&\left[\displaystyle\sum_{K\in\cT_h^f}\left(\textbf{R}_K^f(\textbf{U}_h),\textbf{V}-\textbf{V}_h)\right)_{K}+\displaystyle\sum_{K\in\cT_h^p}\left(\textbf{R}_K^p(\textbf{U}_h),\textbf{V}-\textbf{V}_h)\right)_{K}\right],
	\end{eqnarray}
	where
	\begin{eqnarray}\nonumber\label{errorequations1}
	\left(\textbf{R}_K^f(\textbf{U}_h),\textbf{V}-\textbf{V}_h\right)_K&=&
	\left(\textbf{f}_{fh}+2\nu\nabla\cdot\mathbb{D}(\textbf{u}_{fh})-\nabla p_h,\textbf{v}_{f}-\textbf{v}_{fh}\right)_K-\left(q,\nabla\cdot\textbf{u}_{fh}\right)_K+
	(\textbf{f}_f-\textbf{f}_{fh},\textbf{v}_{f}-\textbf{v}_{fh})_K\\\nonumber
	&-&\left[\displaystyle\sum_{E\in\cE_h(\partial K\cap \Gamma)}\displaystyle
	\sum_{j=1}^{d-1}\left(2\nu\textbf{n}_f\cdot\mathbb{D}(\textbf{u}_{fh})\cdot\tau_j+
	\frac{\alpha}{\sqrt{\tau_j\cdot\textbf{K}\cdot\tau_j}}\textbf{u}_{fh}\cdot\tau_j,(\textbf{v}_{f}-\textbf{v}_{fh})\cdot\tau_j\right)_E\right]\\
	&+&\displaystyle\left[\sum_{E\in\cE_h(\partial K\cap \Gamma)}\left(
	p_h-2\nu\textbf{n}_f\cdot\mathbb{D}(\textbf{u}_{fh})\cdot\textbf{n}_{f}-\rho g \phi_h,(\textbf{v}_f-\textbf{v}_{fh})\cdot\textbf{n}_f\right)_E\right],
	\end{eqnarray}
	and 
	\begin{eqnarray}\nonumber
	\left(\textbf{R}_K^p(\textbf{U}_h),\textbf{V}-\textbf{V}_h\right)_K&=&
	\left(\curl (\rho g\textbf{K}^{-1}\textbf{u}_{ph}+\nabla\phi_h),\beta_p-\beta_{ph}\right)_K\\\nonumber
	&+&\left(\rho g(f_{ph}-\nabla\cdot\textbf{u}_{ph}),\psi\right)_K+\rho g(f_p-f_{ph},\psi)_K\\\nonumber
	&-&\displaystyle\sum_{E\in\cE_h(\partial K\cap \Omega_p)}
	\left(\rho g \textbf{K}^{-1}\textbf{u}_{ph}+\nabla\phi_h)\times \textbf{n}_E,\psi\right)_E\\
	\label{errorequationd1}
	&+&\displaystyle\sum_{E\in\cE_h(\partial K\cap \Omega_p)}\left([\rho g\phi_h\textbf{n}_E]_E,\beta_p-\beta_{ph}\right)_E\\\nonumber
	&-&\displaystyle\sum_{E\in\cE_h(\partial K\cap \Gamma)}\left([\rho g\phi_h\textbf{n}_E]_E,[\textbf{v}-\textbf{v}_h]_E\right)_E.
	\end{eqnarray}
	The inf-sup condition of $\mathcal{L}_h$ leads to:
	$$\parallel\textbf{U}-\textbf{U}_h\parallel\leq C \displaystyle\sup_{\textbf{V}\in \textbf{H}}\frac{|\mathcal{L}_h(\textbf{U}-\textbf{U}_h,\textbf{V})|}{\parallel\textbf{V}\parallel}.$$
	Now, using the error equation (\ref{errorequation1})-(\ref{errorequationd1}), Cauchy-Schwarz inequality and the Cl\'ement operator of lemma \ref{lem1} we deduce the estimate (\ref{reliability}). The proof is complete.
\end{proof}
\subsubsection{Efficiency result}
To prove local efficiency for $w\subset \Omega$, let us denote by
\begin{eqnarray*}
\parallel(\textbf{v}_f,\textbf{v}_p)\parallel_{h,w}^2&:=&\displaystyle\sum_{K\subset \bar{w}\cap \bar{\Omega}_f}|\textbf{v}_f|_{1,K}^2\\
&+&\displaystyle\sum_{K\subset \bar{w}\cap\bar{\Omega}_p}
\left(\parallel \textbf{v}_p\parallel_K^2+\parallel\dive_h\textbf{v}_p\parallel_K^2\right)\\
&+& \parallel \textbf{v}_f\times \textbf
n\parallel_{\Gamma\cap \bar{w}}^2+\displaystyle\sum_{K\subset \bar{\Omega}}\textbf{J}_K(\textbf{v}_f,\textbf{v}_f),
\end{eqnarray*}
where
$$
\textbf{J}_K(\textbf{v}_f,\textbf{v}_f):=\displaystyle\sum_{E\in\cE_h(\bar{\Omega}_f)\cap \cE_h(K)}\delta h_E^{-1}\parallel [\textbf{v}_f]_E\parallel_E^2;
$$
and 
\begin{eqnarray*}
	\parallel (\epsilon,\lambda)\parallel_w:=\parallel\epsilon\parallel_w+\parallel\lambda\parallel_w.
\end{eqnarray*} 
The main result of this subsection can be stated as follows
\begin{thm}[\textbf{Efficiency of $\Theta$}]
	Under the assumptions of Theorem \ref{treg}, the following lower error bound holds:
	\begin{eqnarray*}
		\Theta_K&\lesssim& \parallel (\textbf{e}_f,\textbf{e}_p)\parallel_{h,\tilde{w}_K}+
		\parallel (\epsilon_p,\lambda_{\phi})\parallel_{\tilde{w}_K}+\displaystyle\sum_{K'\subset \tilde{w}_K} \zeta_{K'},
	\end{eqnarray*}
where $\tilde{w}_K$ is a finite union of neighboring elements of $K$.
\end{thm}
\begin{proof}
We begin by bounding each term of the residuals separately.
\\
$\bullet$ Element residual in $\Omega_f$: To estimate $h_K^2\parallel \textbf{f}_{fh}+2\nu\nabla\cdot\mathbb{D}(\textbf{u}_{fh})-\nabla p_h\parallel_K^2$, we choose in error equation (\ref{errorequation1})-(\ref{errorequationd1}) for each $K\in\cT_h^f$, $\textbf{V}=(\textbf{v}_s^K,0,\textbf{v}_p^K,0)$ and $\textbf{V}_h=(0,0,0,0)$ with 
$\textbf{v}_p^K=0$ on $\Omega_p$, 
$$\textbf{v}_f^K=
\left\{
\begin{array}{cccccccccc}
\left[\textbf{f}_{fh}+2\nu\nabla\cdot \mathbb{D}(\textbf{u}_{fh})-\nabla p_h\right]b_K \mbox{  on  } K\in\cT_h^f\\
0 \mbox{  on  } \Omega_f\smallsetminus K
\end{array}
\right.
$$ for obtained,
$
	\mathcal{L}_h(\textbf{U}-\textbf{U}_h,\textbf{V})=
	\parallel\left[\textbf{f}_{fh}+2\nu\nabla\cdot \mathbb{D}(\textbf{u}_{fh})-\nabla p_h\right]b_K^{1/2} \parallel_K^2+(\textbf{f}_f-\textbf{f}_{fh},\textbf{v}_f^K)_K.
$
Noted that:\\ $\left[\textbf{f}_{fh}+2\nu\nabla\cdot \mathbb{D}(\textbf{u}_{fh})-\nabla p_h\right]b_K\in \mathbb{P}^k(K)$ and 
vanish on $\partial K$.
Because, $\mathcal{L}_h(\textbf{U}-\textbf{U}_h,\textbf{V})=(\mathbb{D}(\textbf{e}_f), \textbf{v}_f^K)_K$ in this case, then we have:
 $
 (\mathbb{D}(\textbf{e}_f), \textbf{v}_f^K)_K=
 	\parallel\left[\textbf{f}_{fh}+2\nu\nabla\cdot \mathbb{D}(\textbf{u}_{fh})-\nabla p_h\right]b_K^{1/2} \parallel_K^2+(\textbf{f}_f-\textbf{f}_{fh},\textbf{v}_f^K)_K.
 $
The first inverse inequality (\ref{cl1}) and the Cauchy-Schwarz inequality lead to
\begin{eqnarray*}
\parallel\textbf{f}_{fh}+2\nu\nabla\cdot \mathbb{D}(\textbf{u}_{fh})-\nabla p_h
\parallel_K^2&\sim& 
	\parallel\left[\textbf{f}_{fh}+2\nu\nabla\cdot \mathbb{D}(\textbf{u}_{fh})-\nabla p_h\right]b_K^{1/2} \parallel_K^2\\
	&=&\int_K\mathbb{D}(\textbf{e}_f)\left[\textbf{f}_{fh}+2\nu\nabla\cdot \mathbb{D}(\textbf{u}_{fh})-\nabla p_h\right]b_K-(\textbf{f}_f-\textbf{f}_{fh},\textbf{v}_f^K)_K\\
	&\leq& |\textbf{e}_f|_K|(\textbf{f}_{fh}+2\nu\nabla\cdot \mathbb{D}(\textbf{u}_{fh})-\nabla p_h)b_K|_{1,K}+\parallel\textbf{f}_f-\textbf{f}_{fh}\parallel_K\parallel\textbf{v}_f^K\parallel_K
\end{eqnarray*}
Using inverse inequality (\ref{cl2}) we deduce the estimate:
\begin{eqnarray}\label{est1}
	h_K\parallel\textbf{f}_{fh}+2\nu\nabla\cdot \mathbb{D}(\textbf{u}_{fh})-\nabla p_h \parallel_K&\lesssim& 
	\parallel(\textbf{e}_f,\textbf{e}_p)\parallel_{h,\tilde{w}_K}+\zeta_K.
\end{eqnarray}
$\bullet$ Divergence element residual in $\Omega_f$ (Estimation of $\parallel\nabla\cdot\textbf{u}_{fh}\parallel_K^2$):  For each $K\in\cT_h^f$, we have,
\begin{eqnarray}\nonumber\label{est2}
\parallel \nabla\textbf{u}_{fh}\parallel_K&=&
\parallel\nabla\left( \textbf{u}_f-\textbf{u}_{fh}\right)\parallel_K\\
&\lesssim&|\textbf{u}_f-\textbf{u}_{fh}|_{1,K}.
\end{eqnarray}
$\bullet$ Element residual in $\Omega_p$: We have for each $K\in\cT_h^p$,
\begin{eqnarray}\nonumber
	\parallel\rho g(f_{ph}-\nabla\cdot \textbf{u}_{ph})\parallel_K&=&\parallel	\rho g(f_{ph}-\nabla\cdot \textbf{u}_{ph})-\rho g(f_p-\nabla\cdot \textbf{u}_p)\parallel_K\\\nonumber
	&=&\parallel\rho g \dive_h(\textbf{u}_p-\textbf{u}_{ph})-\rho g(f_p-f_{ph})\parallel_K
	\\
	&\lesssim& \parallel (\textbf{e}_f,\textbf{e}_p)\parallel_{h,\tilde{w}_K}+\zeta_K.
\end{eqnarray}
$\bullet$ Curl element residual in $\Omega_p$: For $K\in \cT_h^p$, we set $C_K=\curl (\rho g\textbf{K}^{-1}\textbf{u}_{ph}+\nabla\phi_h)$ and \\$\textbf{W}_K=C_Kb_K$. Hence we notice that $\curl \textbf{W}_K$ belongs to $\textbf{H}$ and is divergence free, therefore by equation (\ref{errorequation1})-(\ref{errorequationd1}), we obtain with $\textbf{V}_h=\textbf{0}$ and $\beta_p=\textbf{W}_K$,
$\mathcal{L}_h(\textbf{U}-\textbf{U}_h,\textbf{W}_K)=\left(\textbf{R}_K^p(\textbf{U}_h),\textbf{W}_K\right)_K=\parallel \curl (\rho g \textbf{K}^{-1}\textbf{u}_{ph}+\nabla \phi_h)b_K^{1/2}\parallel_K^2$.
The first inverse inequality (\ref{cl1}) and the Cauchy-Schwarz inequality lead to
\begin{eqnarray*}
\parallel \curl (\rho g \textbf{K}^{-1}\textbf{u}_{ph}+\nabla \phi_h)\parallel_K^2 &\sim&	\parallel \curl (\rho g \textbf{K}^{-1}\textbf{u}_{ph}+\nabla \phi_h)b_K^{1/2}\parallel_K^2\\
&=&\int_K\left[\rho g\textbf{K}^{-1}(\textbf{u}_p-\textbf{u}_{ph})+\nabla (\phi-\phi_h)\right]\cdot\curl \textbf{W}_K\\
&\leq& \left(\parallel\rho g\textbf{K}^{-1}(\textbf{u}_p-\textbf{u}_{ph})\parallel_K+\parallel\nabla (\phi-\phi_h)\parallel_K\right)\cdot\parallel\curl \textbf{W}_K\parallel_K
\\
&\lesssim& \left(\parallel\rho g\textbf{K}^{-1}(\textbf{e}_p) \parallel_K+\parallel\nabla (\lambda_{\phi})\parallel_K\right)\cdot
\parallel\textbf{W}_K\parallel_K.
\end{eqnarray*}
Again the inverse inequality (\ref{cl1}) allows to get:
\begin{eqnarray}\label{est3}
	\parallel \curl (\rho g \textbf{K}^{-1}\textbf{u}_{ph}+\nabla \phi_h)\parallel_K
	&\lesssim&
	\parallel (\textbf{e}_f,\textbf{e}_p)\parallel_{h,\tilde{w}_K}+\parallel(\epsilon_p,\lambda_{\phi})\parallel_{\tilde{w}_K}
\end{eqnarray}
$\bullet$ Interface elements on $\Gamma$: We fix an  edge $E$ included in $\Gamma$ and for a constant $\textbf{r}_E$ fixed later on an a unit vector $\textbf{N}$, we consider
$$\textbf{W}_E=\textbf{r}_Eb_E\textbf{N},$$
that clearly belongs to $\textbf{H}$. Hence by residual equation (\ref{errorequation1})-(\ref{errorequationd1}) we obtain with $\textbf{V}_h=\textbf{0}$,
\begin{eqnarray}\label{id1}
\mathcal{L}_h(\textbf{U}-\textbf{U}_h,\textbf{W}_E)=\left(\textbf{R}_{K_f}^f(\textbf{U}_h),\textbf{W}_E\right)_{K_f}+\left(\textbf{R}_{K_p}^p(\textbf{U}_h),\textbf{W}_E\right)_{K_p},
\end{eqnarray}
where $K_f$ (resp. $K_p$) is the unique triangle/tetrahedron included in $\bar{\Omega}_f$ (resp. $\bar{\Omega}_p$) having $E$ as edge/face, and 
\begin{eqnarray}\label{id2}
\left(\textbf{R}_{K_f}^f(\textbf{U}_h),\textbf{W}_E\right)_{K_f}&=&
\left(\textbf{f}_{f}+2\nu\nabla\cdot\mathbb{D}(\textbf{u}_{fh})-\nabla p_h,\textbf{W}_E\right)_{K_f}-\left(q,\nabla\cdot\textbf{u}_{fh}\right)_{K_f}\\\nonumber
&-&\left[\displaystyle\sum_{E\in\cE_h(\partial K_f\cap \Gamma)}\displaystyle
\sum_{j=1}^{d-1}\left(2\nu\textbf{n}_f\cdot\mathbb{D}(\textbf{u}_{fh})\cdot\tau_j+\frac{\alpha}{\sqrt{\tau_j\cdot\textbf{K}\cdot\tau_j}}\textbf{u}_{fh}\cdot\tau_j,\textbf{W}_{E}\cdot\tau_j\right)_E\right]\\\nonumber
&+&\displaystyle\left[\sum_{E\in\cE_h(\partial K_f\cap \Gamma)}\left(
p_h-2\nu\textbf{n}_f\cdot\mathbb{D}(\textbf{u}_{fh})\cdot\textbf{n}_{f}-\rho g \phi_h,\textbf{W}_E)\cdot\textbf{n}_f\right)_E\right],
\end{eqnarray}
and 
\begin{eqnarray}\nonumber\label{id3}
\left(\textbf{R}_K^p(\textbf{U}_h),\textbf{W}_E\right)_{K_p}&=&
\left(\curl (\rho g\textbf{K}^{-1}\textbf{u}_{ph}+\nabla\phi_h),\beta_p\right)_{K_p}\\\nonumber
&+&\left(\rho g(f_p-\nabla\cdot\textbf{u}_{ph}),\psi\right)_{K_p}\\\nonumber
&-&\displaystyle\sum_{E\in\cE_h(\partial K_p\cap \Omega_p)}
\left(\rho g \textbf{K}^{-1}\textbf{u}_{ph}+\nabla\phi_h)\times \textbf{n}_E,\psi\right)_E\\
&+&\displaystyle\sum_{E\in\cE_h(\partial K\cap \Omega_p)}\left([\rho g\phi_h\textbf{n}_E]_E,\beta_p\right)_E\\\nonumber
&-&\displaystyle\sum_{E\in\cE_h(\partial K\cap \Gamma)}\left([\rho g\phi_h\textbf{n}_E]_E,[\textbf{W}]_E\right)_E.
\end{eqnarray}
$\star$ Taken $\textbf{W}_E =\textbf{0} $ in $K_p$, $q=0$ in $\Omega$ and for each
 $j=1,\ldots,d-1$, $\textbf{r}_E=
2\nu\textbf{n}_f\cdot\mathbb{D}(\textbf{u}_{fh})\cdot \tau_j+\frac{\alpha}{\sqrt{\tau_j\cdot\textbf{K}\cdot\tau_j}} \textbf{u}_{fh}\cdot\tau_j$ with $\textbf{N}=\tau_j$. We have 
\begin{eqnarray*}
	\mathcal{L}_h(\textbf{U}-\textbf{U}_h,\textbf{W}_E)&=&
	\left(\textbf{f}_{f}+2\nu\nabla\cdot\mathbb{D}(\textbf{u}_{fh})-\nabla p_h,\textbf{W}_E\right)_{K_f}\\\nonumber
	&-&\left(2\nu\textbf{n}_f\cdot\mathbb{D}(\textbf{u}_{fh})\cdot\tau_j+\frac{\alpha}{\sqrt{\tau_j\cdot\textbf{K}\cdot\tau_j}}\textbf{u}_{fh}\cdot\tau_j,\textbf{W}_{E}\cdot\tau_j\right)_E
	\\
	&=&\left(\textbf{f}_{f}+2\nu\nabla\cdot\mathbb{D}(\textbf{u}_{fh})-\nabla p_h,\textbf{W}_E\right)_{K_f}-
	\parallel\textbf{r}_Eb_E^{1/2}\parallel_E^2
	\end{eqnarray*}
Hence, 
\begin{eqnarray*}
\parallel\textbf{r}_E\parallel_E^2\sim \parallel\textbf{r}_Eb_E^{1/2}\parallel_E^2&=&
\left(\textbf{f}_{f}+2\nu\nabla\cdot\mathbb{D}(\textbf{u}_{fh})-\nabla p_h,\textbf{W}_E\right)_{K_f}-
\mathcal{L}_h(\textbf{U}-\textbf{U}_h,\textbf{W}_E)\\
&=&
\left(\textbf{f}_{fh}+2\nu\nabla\cdot\mathbb{D}(\textbf{u}_{fh})-\nabla p_h,\textbf{W}_E\right)_{K_f}
-
\int_{K_f}2\nu \mathbb{D}(\textbf{e}_f):\mathbb{D}(\textbf{W}_E)\\
&+&\int_{K_f}\epsilon_p\dive \textbf{W}_E+(\textbf{f}_f-\textbf{f}_{fh},\textbf{W}_E)_{K_f}-\displaystyle\sum_{j=1}^{d-1}\frac{\alpha}{\sqrt{\tau_j\cdot\textbf{K}\cdot\tau_j}}
(\textbf{e}_f\cdot\tau_j,\textbf{W}_E\cdot\tau_j)_E.
\end{eqnarray*}
Inverse inequalities (\ref{cl2})-(\ref{cl3}) and Cauchy-Schwarz inequality lead to:
\begin{eqnarray}\label{est4}
	h_E^{\frac{1}{2}}
	\parallel 2\nu\textbf{n}_f\cdot\mathbb{D}(\textbf{u}_{fh})\cdot\tau_j+\frac{\alpha}{\sqrt{\tau_j\cdot\textbf{K}\cdot\tau_j}}\textbf{u}_{fh}\cdot\tau_j\parallel_E\lesssim
	\parallel(\textbf{e}_f,\textbf{e}_p)\parallel_{h,w_E}+\parallel (\epsilon_p,\lambda_{\phi})\parallel_{w_E}+\displaystyle\sum_{K\subset w_E}\zeta_K,
\end{eqnarray}
with $w_E=K_f\cup K_p$.\\\\
$\star$Taken $\textbf{W}_E =\textbf{0} $ in $K_f$, $q=0$ in $\Omega$ and for each
$j=1,\ldots,d-1$, $\textbf{r}_E=
2\nu\textbf{n}_f\cdot\mathbb{D}(\textbf{u}_{fh})\cdot \tau_j+\frac{\alpha}{\sqrt{\tau_j\cdot\textbf{K}\cdot\tau_j}} \textbf{u}_{fh}\cdot\tau_j$ with $\textbf{N}=\textbf{n}_f$. As before the identities (\ref{id1})-(\ref{id3}) and the inverse inequalities (\ref{cl2}) and (\ref{cl3}) lead to
\begin{eqnarray}\label{est5}
	h_E^{\frac{1}{2}}\parallel p_h-2\nu\textbf{n}_f\cdot\mathbb{D}(\textbf{u}_{fh})\cdot
	\textbf{n}_f-\rho g\phi_h\parallel_E\lesssim \parallel(\textbf{e}_f,\textbf{e}_p)\parallel_{h,w_E}+\parallel (\epsilon_p,\lambda_{\phi})\parallel_{w_E}+\displaystyle\sum_{K\subset w_E}\zeta_K.
\end{eqnarray}
$\bullet$ Piezometric head jump in $\Omega_p$: For each edge/face $E\in\cE_h(\Omega_p)$, we consider $w_E=T_1\cup T_2$. As $[\rho g \phi_h \textbf{n}_p]_E\in[\mathbb{P}^{1}(E)]^d$ we set 
$$\textbf{W}_E:=[\rho g \phi_h \textbf{n}_p]_Eb_E\in [H_0^1(w_E)]^d.$$
Using the residual equation (\ref{errorequation1})-(\ref{errorequationd1}) we obtain with $\textbf{V}_h=\textbf{0}$ and $\textbf{V}=(0,0,\beta_p,0)$ where $\beta_p=\textbf{W}_E$:
\begin{eqnarray*}
	\mathcal{L}_h(\textbf{U}-\textbf{U}_h,\textbf{V})=([\rho g\phi_h\textbf{n}_E]_E,\textbf{W}_E)_E.
\end{eqnarray*}
Therefore, Cauchy-Schwarz inequality and inverse inequalities (\ref{cl3})-(\ref{cl4}) lead to:
\begin{eqnarray*}
\parallel[\rho g\phi_h\textbf{n}_E]_E\parallel_E^2\sim\parallel[\rho g\phi_h\textbf{n}_E]_Eb_E^{\frac{1}{2}}\parallel_E^2&=&\int_E[\rho g\phi_h\textbf{n}_E]_E\cdot\textbf{W}_E=\displaystyle\sum_{i=1}^2\rho g(\phi-\phi_h,\nabla\cdot \textbf{W}_E)_{K_i}\\
&\leq& \displaystyle\sum_{i=1}^2\parallel\rho g \lambda_{\phi}\parallel_{K_i}\parallel\nabla\cdot\textbf{W}_E\parallel_{K_i}\\
&\lesssim& \displaystyle\sum_{i=1}^2\parallel\lambda_{\phi}\parallel_{K_i}h_E^{-1}\parallel[\rho g\phi_h\textbf{n}_E]_E\parallel_E.
\end{eqnarray*}
In additionally, since by regularity Theorem \ref{treg} the jump of $\textbf{u}$ is zero through all the edges of $\Omega$, hence we clearly have $\textbf{J}_{\Gamma}(\textbf{U}-\textbf{U}_h,\textbf{U}-\textbf{U}_h)=\textbf{J}_{\Gamma}(\textbf{U}_h,\textbf{U}_h)$. Thus,
\begin{eqnarray}\label{est6}
h_E^{1/2}\parallel[\rho g \phi_h\textbf{n}_p]_E\parallel_E+\frac{\delta h_E^{1/2}}{h}
\parallel[(\textbf{u}_{fh}-\textbf{u}_{ph})\cdot\textbf{n}_f]_E\parallel_E\lesssim
\parallel(\textbf{e}_f,\textbf{e}_p)\parallel_{h,w_E}+\parallel (e_p,\lambda_{\phi})\parallel_{w_E}.
\end{eqnarray}
$\bullet$ Finally, for $E\in\cE_h(\partial K\cap \Omega_p)$, we have 
\begin{eqnarray*}
	\textbf{K}^{-1}\textbf{u}_{ph}+\nabla \phi_h&=&\textbf{K}^{-1}\textbf{u}_{ph}+\nabla\phi_h-\textbf{K}^{-1} \textbf{u}_p-\nabla \phi\\
	&=&-\left[\textbf{K}^{-1}(\textbf{u}_p-\textbf{u}_{ph})+\nabla(\phi-\phi_h)\right]
\end{eqnarray*}
Thus,
\begin{eqnarray}\label{est7}
	\parallel[(\textbf{K}^{-1}\textbf{u}_{ph}+\nabla \phi_h)\times \textbf{n}_p]_E\parallel_E &\lesssim& 
	\parallel(\textbf{e}_f,\textbf{e}_p)\parallel_{h,K}+\parallel(\epsilon_p,\lambda_{\phi})\parallel_K
\end{eqnarray}
The estimates (\ref{est1}), (\ref{est2}), (\ref{est3}), (\ref{est4}), (\ref{est5}), (\ref{est6}) and (\ref{est7})  provide the desired local lower error bound.
\end{proof}
\section{Summary}
In this paper we have discussed a posteriori error estimates for a finite element approximation of the Stokes-Darcy system. A residual type a posteriori error estimator is provided, that is both reliable and efficient. Many issues remain to be addressed in this area,let us mention other types of  a posteriori error estimators or implementation and convergence analysis of adaptive finite element methods.
\section{Nomenclatures}
\begin{itemize}
	\item $\Omega\subset \mathbb{R}^d, d\in\{2,3\}$ bounded domain 
	\item $\Omega_p:$ the porous medium domain
	\item $\Omega_f=\Omega\smallsetminus \overline{\Omega}_d$
	\item $\Gamma=\partial \Omega_f\cap\partial \Omega_p$
	\item $\Gamma_l=\partial \Omega_l\smallsetminus\Gamma,$ $l=f,p$
	\item $\textbf{n}_f$ (resp. $\textbf{n}_p$) the unit outward normal vector along $\partial \Omega_f$ (resp. $\partial \Omega_p$)
	\item $\textbf{u}$: the fluid velocity
	\item $p$: the fluid pressure
	\item 
	In $2D$, the $\curl$ of a scalar function $w$ is given as usual by 
	$$\curl w:=\left(\frac{\partial w}{\partial x_2},-\frac{\partial w}{\partial x_1}\right)^{\top}$$
	\item 
	In $3D$, the $\curl $ of a vector function $\textbf{w}=(w_1,w_2,w_3)$
	is given as usual by $\curl \textbf{w}:=\nabla \times \textbf{w}$ namely,
	\begin{eqnarray*}
		\curl \textbf{w}&:=& 
		\left(\frac{\partial w_3}{\partial x_2}-\frac{\partial w_2}{\partial x_3},
		\frac{\partial w_1}{\partial x_3}-\frac{\partial w_3}{\partial x_1},
		\frac{\partial w_2}{\partial x_1}-\frac{\partial w_1}{\partial x_2}\right)
	\end{eqnarray*}
	\item $\mathbb{P}^k$: the space of polynomials of total degree not larger than $k$
	\item $\cT_h$: triangulation of $\Omega$
	\item $\cT_h^l$: the corresponding induced triangulation of $\Omega_l$, $l\in\{f,p\}$
	\item For any $K\in\cT_h$, $h_K$ is the diameter of $K$ and $\rho_K=2r_K$ is the diameter of the largest ball inscribed into $K$
	\item $h:=\displaystyle\max_{K\in\cT_h} h_K$  and   $\sigma_h:=\displaystyle\max_{K\in\cT_h} \frac{h_K}{\rho_K}$
	\item $\cE_h$: the set of all the edges or faces of the triangulation
	\item $\cE(K)$: the set of all the edges ($N=2$) or faces ($N=3$) of a element $K$
	\item $\cE_h:=\displaystyle\bigcup_{K\in\cT_h} \cE(K)$
	\item $\mathcal{N}(K)$: the set of all the vertices of a element $K$
	\item $\mathcal{N}_h:=\displaystyle\bigcup_{K\in\cT_h} \mathcal{N}(K)$
	\item For $\mathcal{A}\subset \overline{\Omega}$, $\mathcal{E}_h(\mathcal{A}):=\{E\in\mathcal{E}_h: E\subset \mathcal{A}\}$
	\item For $E\in\cE_h$, we associate a unit vector $\textbf{n}_E$ such that $\textbf{n}_E$ is orthogonal to $E$ and equals to the unit exterior normal vector to $\partial\Omega$
	\item For $E\in\cE_h$, $[\phi]_E$ is the jump across $E$ in the direction of $\textbf{n}_E$
	\item In order to avoid excessive use of constants, the abbreviations $x\lesssim y$ and $x\sim y$ 
	stand for $x\leqslant cy$ and $c_1x\leqslant y \leqslant c_2x$, respectively, with positive constants independent of $x$, $y$ or $\cT_h$.
\end{itemize}

\newcommand{\noopsort}[1]{}

\end{normalsize}
\end{document}